\documentclass[12pt]{article}
\usepackage{amsmath, amssymb, graphics, hyperref, multirow, subfig, amsthm}
\usepackage[title]{appendix}
\begin{document}
\newtheorem{conjecture}{Conjecture}
\newtheorem{proposition}{Proposition}
\newtheorem{lemma}{Lemma}
\title{Primes of the Form $m^2+1$ and Goldbach's `Other Other' Conjecture}
\author{Jon Grantham \\
\normalsize Institute for Defense Analyses,
Center for Computing Sciences \\ 
\normalsize Bowie, Maryland 20715, 
E-mail: grantham@super.org
\and
Hester Graves \\
\normalsize Institute for Defense Analyses,
Center for Computing Sciences \\ 
\normalsize Bowie, Maryland 20715, 
E-mail:  hkgrave@super.org}
\date{\today}

\maketitle
\abstract{We compute all primes up to $6.25\times 10^{28}$ of the form $m^2+1$. 
Calculations using this list verify, up to our bound, a less famous conjecture of Goldbach. 
We introduce `Goldbach champions' as part of the verification process
and prove conditional results about them, assuming either Schinzel's Hypothesis H or the Bateman-Horn Conjecture.}

\vskip 10pt\noindent

\thispagestyle{empty}

\section{Introduction}

Goldbach's most famous conjecture, `Goldbach's conjecture,' is that every even integer greater than or equal to four is the sum of two primes.  
He also conjectured erroneously that every odd composite number $n$ can be written as $p + x^2$, where $p$ is prime; 
Project Euler \cite{projecteuler} terms this `Goldbach's other conjecture.' The two known counterexamples are $5777$ and $5993$
\cite{hodges}.
Here, we study `Goldbach's other other conjecture.'

In an October 1, 1742 letter to Euler, Goldbach \cite{goldbach} conjectured:
\begin{conjecture}\label{original}{(Goldbach's Other Other Conjecture)} Let $A$ be the set of positive integers for which $a^2+1$ is prime. 
All $a>1$ in $A$ can be expressed as $b+c$, for some $b,c\in A$.
\end{conjecture}
 
In 1912, not long after the proof of the Prime Number Theorem, Landau \cite{pintz} described determining
whether the set $A$ is infinite as ``unattackable at the present state of science.'' More than a century later,
the problem still resists all attempts. Accordingly, we present a 
computational algorithm and verification up to a bound.  While explaining heuristics to verify the conjecture, we introduce `Goldbach champions'
in Section \ref{verification}.  

The sequence $A$ is studied in another context. We denote the largest prime factor of an integer $n>1$ by $P(n)$. Pasten
\cite{hector}
recently improved Chowla's 1934 results, showing that $P(n^2+1) \gg \frac{(\log\log n)^2}{\log\log\log n}$. Using this notation,
we restate $|A|=\infty$ as $|\{n:P(n^2+1)=n^2+1\}|=\infty$.

If $m^2 + 1$ is prime, we define $A_m = \{a \leq m: a^2 + 1 \text{ is prime}\}$ and test Conjecture \ref{original} by looking at the differences 
 $\{m -a: a \in A_m\}$.  More precisely, if $|A_m| =n$, we enumerate the 
 elements of $A_x$ as $a_1, \cdots, a_n=m$, where $a_i < a_j$ if $i < j$.  
 We look at the differences 
 $a_n - a_{n-1}, a_n - a_{n-2}, \cdots, a_n - a_1 $ to confirm that $m = a_i + a_j$ for some $1\leq i \leq j \leq n-1$.
 We then ask what is the smallest $i \in [1, n-1]$ such that $a_n - a_{n-i} \in A_m$, and how large $i$ is with respect to $m$.
 We denote this smallest value of $i$ by $j(m)$ and examine it in Section \ref{verification}.  
 In Section \ref{Conditional}, we 
prove results on the values of $j(x)$ and the verification process, one of which is conditional on the Bateman-Horn conjecture and
the others on Schinzel's Hypothesis H.
Appendix A is a table comparing champion values of $j(a_n)$ to $\log n$.

\section{Enumeration of primes of the form $m^2+1$}

Wolf \cite{wolf} computed the primes $p=m^2+1$, for $p<10^{20}$.
Wolf and Gerbicz \cite{oeis} then published a table up to $10^{25}$.  We extended their table up to $6.25\times 10^{28}$.

Our computation uses three sieves, thereby creating three lists of primes.  
Let $B$ be the upper bound of our eventual list of primes $p=m^2+1$, so that $p <B$.
In Wolf's original article, $B = 10^{20}$.  
The first step uses the Sieve of Erathosthenes to generate the primes up to $B^{1/4}$.
Our second list starts as all positive integers $z < B^{\frac{1}{2}}$, $z \equiv 1 \pmod{4}$.
We then sieve using our first list of primes, so our second list becomes the set of all primes $p < B^{\frac{1}{2}}$, $p \equiv 1 \pmod{4}$.
We also compute the roots of $-1$ modulo $p$ for every prime on the second list, and store them with said primes.

We now use our second list of primes to perform the third sieve, on all positive integers $x \leq B^{1/2}$.  
If $B > x^2 +1 > B^{\frac{1}{2}}$, then $x^2 +1$ is prime if and only if $x$ is not a square root of $-1$ modulo any of the primes in the second list.
We therefore use the second list of primes, with the accompanying list of square roots of $-1$, to list the primes of the form $m^2 +1$.

From \cite{cp}, p.\ 121, the number of operations to sieve an array of length $A$ with the primes up to $P$ is 
\begin{equation}O(A\log\log{P}+P^{1/2}/\log{P}).\label{eq:cranpom}\end{equation}
The first summand represents the required sieve updates, and is mostly determined by the size of the array. 
The second term represents the per-prime work to find the sieve starting location, and depends on the number of primes. 
For longer arrays, the first term dominates.  For shorter arrays, the second does.
 
First, we consider our algorithm's computational complexity, assuming the entire sieve array fits into memory.  
We see later that our real-world conditions are more complicated,
but it makes sense to start off analyzing the triple sieve's computational complexity without adding the extra hardware restrictions.
The first sieve (of Eratosthenes) (up to $B^{1/4}$) takes $O(B^{1/4}\log\log{B})$ operations, and the second sieve (up to $B^{1/2}$) takes $O(B^{1/2}\log\log{B})$ operations.
Computing the roots of $-1$ requires computing $2^{(p-1)/4}\bmod p$ for $O(B^{1/2}/\log{B})$ primes. Each exponentiation takes $O(\log{B})$ operations, so the total work for computing the roots is $O(B^{1/2})$.
Thus the entirety of the work done before embarking on the third sieve is $O(B^{1/2}\log \log B)$.

The third sieve is somewhat unusual in that both the sieve array length and the size of the largest prime is about $B^{1/2}$.
The sieve length is the same as in the second sieve, potentially surprising some readers. This is because the elements of $A$ are
the square roots of one less than the primes.
 The number of operations is $O(B^{1/2}\log\log{B}+B^{1/4}/\log{B})$, 
which is again $O(B^{1/2}\log\log{B})$. Therefore, the overall running time of our triple sieve is $O(B^{1/2}\log\log{B})$.

We were ambitious and decided to find all primes $p$ of the form $m^2 +1$ with $p < 6.25 \times 10^{28}$ (see section \ref{verification}).
Unfortunately, a single sieve with $B =6.25 \times 10^{28}$ would require tens of terabytes of memory to store the two arrays for the second and third sieves, which would be infeasible on virtually all modern machines.
Let $M$ denote the maximum length of a sieve array that can fit into memory. Then we require $\frac{B^{1/2}}M$ instances of the second and third sieves, and find ourselves very familiar with our file system. 

The second sieve is essentially a Sieve of Eratosthenes. 
In the regular Sieve of Eratosthenes, we easily save a factor of two on memory by skipping even numbers; here we save a factor of four by also skipping numbers that are $3$ mod $4$. 
We load our length-$M$ section of our sieve array, and then we sieve by the primes in our first list, i.e., the primes that are $\leq B^{1/4}$.  
For each of the primes $p$ in our first list, we need to find where the first multiple of $p$ is in our array of length $M$, before we proceed to sieve by $p$; that is an easy modular reduction.
The totality of these reductions is the second term in Equation (\ref{eq:cranpom}).  
Once we have finished sieving our array of length $M$, we save all of the primes congruent to $1 \pmod{4}$ that we have found to our second list, and compute the roots of $-1$ modulo these new, saved primes.  
We then clear our memory and load  our next sieve array of length $M$. 
Note that $M > B^{\frac{1}{4}}$, so we do not need to load in our sieving primes; we only need to load segments of the list that we sieve.

If we use Equation (\ref{eq:cranpom}) to determine 
the time it takes to sieve all of the $\frac{B^{1/2}}{M}$ instances,
we get 
\[O\left (\frac{B^{1/2}}{M}(M \log \log B + B^{1/4}/ \log B )\right ) = O(B^{1/2} \log \log B + B^{3/4}/(M \log B)).\]
The first term in our equation continues to dominate, unless the memory available for
the sieve area drops to $O(\frac{B^{1/4}}{ \log {B} \log\log{B}})$. 
In practice, the sieve area never gets that small.  With our chosen value of $B=6.25\times 10^{28}$, $\frac{B^{1/4}}{\log{B}\log\log{B}}$ is less than sixty thousand. A cursory examination of the constants involved shows that we certainly 
would not fill up multi-gigabyte machines.

Implementing the third sieve is trickier.  
Again, we load our third sieve's array of length $M$.  This time we sieve by our second list of primes, which is much longer than our first list of primes, and the 
list of primes that we are sieving by does not itself fit in our memory.  We therefore load the second list of primes sequentially from a series of files, and  sieve from the loaded list.  
Note that every prime is loaded once, but this action is so much smaller than the number of times we sieve with a given prime that the file loading work is lost in the noise.

It is also trickier to sieve by any given $p$ in our second list of primes.  
We still need to find the first multiple of our prime $p$ in the array of length $M$, but we are actually sieving by the associated roots $ \pm r$ of $-1$.  
For each root $\pm r$, there may exist some $x$ in our array of length $M$ that is less than the array's first multiple of $p$  such that $x \equiv  \pm r$.
We spend much more time on the per-prime computations
(with respect to sieve updates) than we did in the second sieve.
In total, we do up to four operations per prime before we start sieving with it.
That is, however, only a constant multiple and does not affect the asymptotics.  

Equation (\ref{eq:cranpom}) shows that we lose efficiency when the sieve array size drops below $O(\frac{B^{1/4}}{\log{B}\log\log{B}})$ --- in other 
words, when the number of instances exceeds $O(\log{B}\log\log{B})$. 
We used $9000$ instances, and as $B = 6.25 \times 10^{28}$,
one would assume from the given asymptotics that loading our sieve arrays was efficient.
We did not, however, do a detailed analysis involving constants, and we suspect we were either near or past the point at which we lose efficiency. 
As we did not have larger-memory machines available to us at the time, we had no choice but to accept any such loss.

Sample code is available at \url{https://github.com/31and8191/Goldbach1}.

\section{Computational Results}\label{section:computation}

We use Wolf's notation that $\pi_q(x)$ is the number of primes of the form $m^2+1$ up to $x$.

As Wolf notes, Hardy and Littlewood's Conjecture E \cite{hardyl} gives $\pi_q(x)\sim f(x)$, where $f(x) =C_q\frac{\sqrt{x}}{\log{x}}$ and $C_q=1.3728\dots$.  
More precise heuristics give $\pi_q(x)\sim g(x)$, with $g(x) =\frac{C_q}2 \operatorname{li}(\sqrt{x})$. 
In his Table I, Wolf computed the values  of $\pi_q(x)$, $f(x)$, $\pi_q(x)/f(x)$, $g(x)$, and $\pi_q(x)/g(x)$ for $x = 10^a$, where $a$ ranges from $6$ to $20$.  Wolf and Gerbicz  \cite{oeis} then computed the appropriate values for $\pi_q(x)$ when $a$ ranges from $21$ to $25$.  
We repeat {\bf and extend} their results in Table 1.
\begin{table}[h!]
\centering
\begin{tabular}{|| r | r | l |l||}
\hline
$x$ & $\pi_q(x)$ & $\pi_q(x)/f(x)$ & $\pi_q(x)/g(x)$ \\
\hline
$10^{1}$ & $2$ & $1.06080$ & $1.20841$ \\
$10^{2}$ & $4$ & $1.34181$ &  $0.92957$ \\
$10^{3}$ & $10$ & $1.59120$ &  $1.07127$ \\
$10^{4}$ & $19$ & $1.27472$ & $0.91567$ \\
$10^{5}$ & $51$ & $1.35252$ &  $1.04253$ \\
        $10^{6}$ & $112$ & $1.12713$ & $0.91869$ \\
        $10^{7}$ & $316$ & $1.17325$ & $0.99440$ \\
        $10^{8}$ & $841$ & $1.12847$ & $0.98321$ \\
        $10^{9}$ & $2378$ & $1.13516$ & $1.00888$ \\
        $10^{10}$ & $6656$ & $1.11639$ & $1.00696$ \\
        $10^{11}$ & $18822$ & $1.09815$ & $1.00184$ \\
        $10^{12}$ & $54110$ & $1.08909$ & $1.00258$ \\
        $10^{13}$ & $156081$ & $1.07621$ & $0.99805$ \\
        $10^{14}$ & $456362$ & $1.07162$ & $0.99991$ \\
        $10^{15}$ & $1339875$ & $1.06601$ & $0.99984$ \\
        $10^{16}$ & $3954181$ & $1.06116$ & $0.99974$ \\
        $10^{17}$ & $11726896$ & $1.05739$ & $1.00005$ \\
        $10^{18}$ & $34900213$ & $1.05367$ & $0.99991$ \\
        $10^{19}$ & $104248948$ & $1.05058$ & $0.99997$ \\
        $10^{20}$ & $312357934$ & $1.04782$ & $1.00001$ \\
$10^{21}$ & $938457801$ & $1.04529$ & $0.999996$ \\
$10^{22}$ & $2826683630$ & $1.04305$ & $1.000005$ \\
$10^{23}$ & $8533327397$ & $1.04100$ & $0.999998$\\
$10^{24}$ & $25814570672$ & $1.03915$ & $1.000008$ \\
$10^{25}$ & $78239402726$ & $1.03746$ &  $1.000004$ \\
$10^{26}$ & $237542444180$  & $1.03590$ & $1.000003$ \\
$10^{27}$ & $722354138859$ & $1.03447$ &  $1.00000003$ \\
$10^{28}$ & $2199894223892$ & $1.03315$ & $1.00000019$ \\
$6.25\times 10^{28}$ & $5342656862803$ & $1.03217$ & $0.99999976$ \\
\hline
\end{tabular}
\caption{Prime counts}
\label{pi_q}
\end{table}

\section{Verifying Goldbach's Other Other Conjecture} \label{verification}
We confirmed Goldbach's other other conjecture up to $6.25\times 10^{28}$,
i.e., for $a$ up to $2.5 \times 10^{14}$. The list of primes takes up
more than 30 terabytes on disk --- it would be challenging to search through that whole list for each prime to find a difference in our set.

Instead, we asked the following naive questions, and used them to guide our simple verification strategy.
Let $A$ be the set of all $a$ such that $a^2 +1$ is prime and let us enumerate them in order, so $A = \{a_n\}$.
The sets $A_m$ in the introduction are truncations of $A$.
\begin{itemize}
\item Is $a_n-a_{n-1}=a_i$ for some $i$? 
\item How about $a_n-a_{n-2}$? 
\item How far back do you have to go?
\end{itemize}
To tackle these questions, note that Section \ref{section:computation}'s claim that $\pi_q(x) \sim C_q \frac{\sqrt{x}}{\log x}$ is equivalent to saying 
$a_n \sim \frac{2}{C_q} n \log \frac{2n}{C_q}$.

Let $j(a_n)$ be the smallest value of $i$ such that $a_n-a_{n-i}=a_k$ for some $k$. We call $a_n$ a \textbf{Goldbach champion} if $j(a_i)<j(a_n)$ for all $i<n$. The appendix contains a list of all champions for $a_n<2.5\times 10^{14}$.

\section{Conditional results about the growth of $j(n)$}\label{Conditional}

Popular conjectures about prime values of polynomials imply interesting patterns in the distribution of the sequence $a_n$.

\begin{conjecture}{(Schinzel's Hypothesis H \cite{hyph})}
A set of polynomials $f_i(x)$ satisfies the Bunyakovsky condition if there is no $p$ for which $\prod f_i(a)\equiv 0$ for all $a\in\mathbb{F}_p$. 
Under this assumption, the polynomials are simultaneously prime for infinitely many values of $x$.
\end{conjecture}

\begin{proposition}
Assuming Hypothesis H, $j(a_n)>1$ infinitely often.
\end{proposition}
\begin{proof}
Let $f_1(y)=(65y+9)^2+1$ and $f_2(y)=(65y+1)^2+1$.
Since each polynomial has at most $2$ roots, $f_1(a) f_2(a)$ cannot be $0$ for all $a \in \mathbb{F}_p$ when $p \geq 5$.
It is easy to check $f_1(a) f_2(a)$ is not always $0$ for all $a \in \mathbb{F}_p$ when $p$ is either $2$ or $3$.
Our set therefore satisfies the Bunyakovsky condition, and thus the two functions will be simultaneously prime infinitely often, assuming Hypothesis H.
To see that they are {\bf consecutive} primes of the form $x^2+1$, look at the intermediate values.
\begin{align*}
    (65y+3)^2+1 &\equiv 0 \pmod{5} \\
    (65y+5)^2+1 &\equiv 0 \pmod{13} \\
    (65y+7)^2+1 &\equiv 0 \pmod{5}
\end{align*}

The difference $(65y+9)-(65y+1)=8$ is not in $A$, so $j(a_n)>1$ infinitely often.
\end{proof}

We can, in fact, prove a much stronger result if we assume the Bateman-Horn Conjecture \cite[p. 363]{BH}.

The Bateman-Horn Conjecture states that the number of values less than $x$, for which a set of $k$ polynomials satisfying the Bunyakovsky condition is simultaneously prime, 
is proportional to $\frac x{\log^k(x)}$, and gives the proportionality constant, which we will not use.
The Bateman-Horn conjecture strengthens Hypothesis H.

\begin{lemma}
Given a sequence $\{y_n\}$ of density $0$ and a positive integer $k$, there exists a set $\{b_0, b_1, \ldots, b_{k-1}\}$ such that
the $b_i\not\in\{y_i\}$ and the set of polynomials 
$\{f_i=(x^2-b_i)\}$ satisfies the Bunyakovsky condition.
\end{lemma}
\begin{proof}
By the Chinese Remainder Theorem, there is a $b$ such that for each prime $p\le 2k$, $b^2+1\not\equiv 0\bmod p$.
Because the set of $y_i$'s has zero density and the numbers equivalent to $b$ modulo all small primes has positive density, we can
choose a set of $b_i$'s congruent to $b$ which avoids the sequence $\{y_n\}$.
For primes $p \leq 2k$,  all of the $f_i(0)\not\equiv 0$  modulo said primes, and the condition is satisfied.
The product of the $f_i$'s has degree $2k$, and therefore cannot be identically zero modulo any prime $p > 2k$.
\end{proof}

\begin{proposition}
Assuming the Bateman-Horn Conjecture,  \[\limsup_{n\to\infty} j(a_n)=\infty.\]
\end{proposition}
\begin{proof}
We demonstrate that for any $k$, there are infinitely many $a_n$ with $j(a_n)\ge k$. 

The preceding lemma shows we can form a sequence $b_0=0,b_1,...,b_{k-1}$ of elements not in $A$ such that the set $\{f_i(m)=(m-b_i)^2+1\}$ satisfies the Bunyakovsky condition. 

Assuming the Bateman-Horn conjecture, there are asymptotically $c \frac x{\log^k(x)}$ values of $m$ less than $x$,  $c >0$, where all of the polynomials take prime values. 
If there are no other values $d$, where $0<d<b_{k-1}$, such that $(m-d)^2+1$ is prime, then we have that $j(m)\ge k$. 

Assume that there are only finitely many $m$ such that there exist no other described $d$. 
Then for all but finitely many $m$, $(m-d)^2+1$ is prime for at least one $d$, with $0<d<b_{k-1}$, that is not equal to any of the $b_i$.
By the pigeonhole principle, at least one of the potential $d$'s creates a $(k+1)$-st polynomial $f_{x+1}(m) = (x-d)^2 +1$ 
such that there are asymptotically $c' \frac{y}{\log^2 y}$ values of $m$ less than $y$, with $c'>0$, such that 
$f_1(m), f_2(m), \ldots , f_{k+1}(m)$ are all simultaneously prime. 
That contradicts the Bateman-Horn conjecture, which says that the set $\{f_i,(m-d)^2+1\}$ can have asymptotic count at most $c'' \frac {y}{\log^{k+1}(y)}$ for some $c'' >0$. 
Therefore there exist infinitely many $m$ such that $j(m) \geq k$.
Our choice of $k$ was arbitrary, so $\limsup_{n \rightarrow \infty} j(a_n) = \infty$.

\end{proof}

\begin{proposition}
Assuming Hypothesis H, $\liminf_{n\to\infty} j(a_n)=1$.
\end{proposition}

\begin{proof}
Consider the polynomials $x^2+1$ and $(x-2)^2+1$. By the above lemma, they satisfy the Bunyakovsky condition.
By Hypothesis H, there are infinitely many $a_i$ with both $a_i^2+1$ and $(a_i-2)^2+1$ prime. Because $(a_i-1)^2+1$ must be even, $a_{i-1}=a_i-2$, and $a_i-a_{i-1}=2$, which is a
member of our set A.
\end{proof}

In particular, this shows that Goldbach's other other conjecture is true infinitely often.

\section{Further Work}

In a follow-up paper, we generalize Goldbach's other other conjecture to cyclotomic polynomials other than $\Phi_4(x)=x^2+1$. 
We thank Michael Filaseta for noting that Goldbach's other other conjecture is equally plausible when looking at representations of all positive integers, not just primes. 
Our forthcoming paper also explores that intriguing path. 

We did not further explore the function $j(a_n)$.  
While it looks like $j(a_n)$ grows infinitely large, we do not have a growth result for $j(a_n)$.  
Arithmetic statisticians may want to explore the `expected value of $j(a_n)$.'
Hypothesis H implies that $j(a_n)$ is one infinitely often, but there may be a nice formula approximating $j(a_n)$ most of the time.

\section{Acknowledgments}

We would like to thank Franz Lemmermeyer for telling the second author about the conjecture, and for sending her a link to Goldbach's correspondence.
Both authors would like to thank Professor Louise Shelley of George Mason University for her assistance with background research. 

We started this research a decade ago, but we were derailed when the second author was seriously injured.  She would like to thank Drs.\ Bullard-Bates, Etsy, and Potolicchio for their help in returning 
her to health, and our center's director at the time, Dr.\ Francis Sullivan, for supporting her during her recovery.  She would also like to thank Ms.\ Joemese Malloy and the Thurgood Marshall Childhood Development 
Center for giving her the peace of mind to engage in research, knowing that her child is in safe and loving hands, especially during the pandemic.  
Lastly, she thanks her husband, Loren LaLonde, for 
his support through all of these trials and tribulations.

\begin{appendices}
\section{Champion values of $j(a)$}
\begin{tabular}{ r | r | r | r|r}
\hline
$n$ &$a_n$ & $\frac{2}{C_q} n \log \frac{2n}{C_q}$ &  $j(a_n)$ & $\frac{j(a_n)}{\log n}$\\
\hline
$16$&$74$ & $106$& $3$ &$1.08$\\
$55$&$384$ & $507$ & $6$ &$1.50$\\
$100$&$860$ & $1047$ & $7$ &$1.52$\\
$173$&$1614$ & $2011$ & $10$&$1.94$ \\
$654$&$7304$ & $9429$ & $12$ &$1.85$\\
$1188$&$14774$ & $18618$ & $14$& $2.00$ \\
$2815$&$37884$ & $49220$ & $17$ &$2.14$\\
$6868$&$103876$ & $132962$ & $21$ &$2.38$\\
$11913$&$191674$ & $244421$ & $23$ &$2.45$\\
$36533$&$651524$ & $835598$ & $24$ &$2.28$\\
$38073$&$681474$ & $874125$ & $26$ &$2.47$\\
$62688$&$1174484$ & $1504969$ & $38$ &$3.44$\\
$480452$&$10564474$ & $13590903$ & $44$ & $3.63$\\
$837840$&$19164094$ & $24679882$ & $48$ & $3.52$\\
$1286852$&$30294044$ & $39066897$ & $52$ &$3.70$\\
$10451620$&$279973066$ & $363307290$ & $56$ &$3.46$\\
$25218976$&$709924604$ & $923322569$ & $58$ &$3.40$\\
$68826857$&$2043908624$ & $2665142759$ & $64$ &$3.55$\\
$79601233$&$2381625424$ & $3106685030$ & $65$ &$3.57$\\
$157044000$&$4862417304$ & $6353414906$ & $69$ &$3.66$\\
$266774400$&$8476270536$ & $11089804641$ & $70$ & $3.61$\\
$337231328$&$10835743444$ & $14184814636$ & $71$ & $3.62$\\
$1702595832$&$58917940844$ & $77409688313$ & $83$ & $3.90$\\
$2524491445$&$88874251714$ & $116867691886$ & $90$ &$4.16$\\
$3079006270$&$109327832464$ & $143823180284$ & $105$ &$4.81$\\
$63281910377$&$2537400897706$ & $3358032936033$ & $125$ & $5.03$\\
\hline
\end{tabular}
\end{appendices}

\bibliography{goldbach}

\begin{thebibliography}{WG10}

\bibitem[BH62]{BH}
Paul~T. Bateman and Roger~A. Horn.
\newblock A heuristic asymptotic formula concerning the distribution of prime
  numbers.
\newblock {\em Math. Comp.}, 16:363--367, 1962.

\bibitem[CP05]{cp}
Richard Crandall and Carl Pomerance.
\newblock {\em Prime Numbers: A Computational Perspective}.
\newblock Springer, New York, second edition, 2005.

\bibitem[Eul]{projecteuler}
Project Euler.
\newblock Goldbach's other conjecture.
\newblock \url{https://projecteuler.net/problem=46}.
\newblock Problem 46.

\bibitem[Fus68]{goldbach}
P.-H. Fuss.
\newblock {\em Correspondance math\'{e}matique et physique de quelques
  c\'{e}l\`ebres g\'{e}om\`etres du {XVIII}\`eme si\`ecle. {T}omes {I}, {II}}.
\newblock Pr\'{e}c\'{e}d\'{e}e d'une notice sur les travaux de L\'{e}onard
  Euler, tant imprim\'{e}s qu'in\'{e}dits et publi\'{e}e sous les auspices de
  l'Acad\'{e}mie Imp\'{e}riale des Sciences de Saint-P\'{e}tersbourg. The
  Sources of Science, No. 35. Johnson Reprint Corp., New York-London, 1968.

\bibitem[HL23]{hardyl}
G.~H. Hardy and J.~E. Littlewood.
\newblock Some problems of `{P}artitio numerorum'; {III}: {O}n the expression
  of a number as a sum of primes.
\newblock {\em Acta Math.}, 44(1):1--70, 1923.

\bibitem[Hod93]{hodges}
Laurent Hodges.
\newblock A {L}esser-{K}nown {G}oldbach {C}onjecture.
\newblock {\em Math. Mag.}, 66(1):45--47, 1993.

\bibitem[Pas24]{hector}
Hector Pasten.
\newblock The largest prime factor of {$n^2+1$} and improvements on
  subexponential {$ABC$}.
\newblock {\em Invent. Math.}, 236(1):373--385, 2024.

\bibitem[Pin09]{pintz}
J\'{a}nos Pintz.
\newblock Landau's problems on primes.
\newblock {\em J. Th\'{e}or. Nombres Bordeaux}, 21(2):357--404, 2009.

\bibitem[SS58]{hyph}
A.~Schinzel and W.~Sierpi{\'n}ski.
\newblock Sur certaines hypoth\`eses concernant les nombres premiers.
\newblock {\em Acta Arith.}, 4:185--208; erratum 5 (1958), 259, 1958.

\bibitem[WG10]{oeis}
Marek Wolf and Robert Gerbicz.
\newblock The {O}n-line {E}ncyclopedia of {I}nteger {S}equences.
\newblock \url{https://oeis.org/A083844}, 2010.
\newblock Number of primes of the form {$x^2 + 1 < 10^n$}.

\bibitem[Wol13]{wolf}
Marek Wolf.
\newblock Some conjectures on primes of the form {$m^2+1$}.
\newblock {\em J. Comb. Number Theory}, 5(2):103--131, 2013.

\end{thebibliography}
\bibliographystyle{alpha}

\end{document}